\documentclass[11pt]{article}
\usepackage{fancyhdr}
\usepackage{isomath}
\usepackage{amsmath}
\usepackage{amsbsy}
\usepackage{amssymb}
\usepackage{amscd}
\usepackage{amsfonts}
\usepackage{array}
\usepackage{color}
\usepackage{graphicx}
\usepackage{verbatim}
\usepackage{euscript}
\usepackage{alltt}
\usepackage{stmaryrd}
\usepackage{caption}
\usepackage{subcaption}
\usepackage{relsize}
\usepackage{enumitem}
\usepackage[numbers]{natbib}
\usepackage{caption}
\usepackage{float}
\usepackage{multirow}
\usepackage{epsfig}
\usepackage{amsthm}
\usepackage{xcolor}

\DeclareGraphicsExtensions{.eps,.pdf,.png}

\usepackage[margin=1in]{geometry}

\newcommand{\curl}{\mbox{curl}\,}

\newcommand{\p}{\partial}
\newcommand{\veps}{\varepsilon}
\newcommand{\lf}[1]{\dot{#1}}
\newcommand{\fo}[1]{\mathbb{#1}}
\newcommand{\tho}[1]{\mathsf{#1}}
\newtheorem{theorem}{Theorem}[section]
\newtheorem{corollary}{Corollary}[theorem]

\date{}
\begin{document}
\title{Existence, uniqueness, and long-time behavior of linearized field dislocation dynamics
}

\author{Amit Acharya\thanks{Department of Civil \& Environmental Engineering, and Center for Nonlinear Analysis, Carnegie Mellon University, Pittsburgh, PA 15213, email: acharyaamit@cmu.edu.} $\qquad$ Marshall Slemrod\thanks{Department of Mathematics, University of Wisconsin, Madison, WI 53706, email: slemrod@math.wisc.edu.}
}
\maketitle
\textit{This paper is dedicated to our friend Costas Dafermos on the occasion of his 80th birthday.}

\begin{abstract}
 \noindent This paper examines a system of partial differential equations describing dislocation dynamics in a crystalline solid. In particular we consider dynamics linearized about a state of zero stress and use linear semigroup theory to establish existence, uniqueness, and time asymptotic behavior of the linear system.
 
 
\end{abstract}

\section{Introduction}
Dislocations are topological defects in elasticity and their dynamics and interaction are of significant scientific and technological interest. They were introduced in the theory of elasticity by Volterra in 1907 and continue to be vigorously studied to this day, with a complete dynamical theory within continuum mechanics of unrestricted material and geometric nonlinearity taking shape in recent years (see, e.g.,~\cite{acharya2022action,acharya2019structure,zhang2015single,arora2020unification,arora_acharya_ijss,arora_acharya_ijss} with detailed bibliographic threads to earlier works in the overall subject). Here, we give a mathematical analysis of the governing nonlinear pde system of dislocation mechanics, linearized about a state of vanishing stress.

This paper has two additional sections after this Introduction. Section \ref{sec:lin_disloc} presents the `small deformation' model of nonlinear dislocation mechanics, see e.g.~\cite{acharya2022action}. We then linearize this system about a state of vanishing stress to obtain a linear system of partial differential equations which describe the evolution of a class of deformations of an elastic solid. In Section \ref{sec:anal} we use the Lumer-Phillips Theorem \cite{yosida} to establish existence and uniqueness of solutions to our linear evolutionary system. Specifically, the Lumer-Phillips Theorem yields existence of a $C^0$ semigroup of contractions on a Hilbert space $H$. Furthermore, this evolution is dissipative and allows us to determine the time-asymptotic behavior via an application of a theorem of S. Foguel \cite{foguel1966ergodic}. Here, time asymptotic behavior is only given with respect to weak convergence in the Hilbert space $H$ as lack of a compact resolvent for the infinitesimal generator of our $C^0$ semigroup appears to preclude application of a theorem of Dafermos and Slemrod \cite{dafermos1973asymptotic} for strong convergence. Finally, two examples are presented which show that the limit dynamics predicted by the linearized \textit{dissipative} theory allows for both a non-trivial static solution as well as an oscillating motion.

\section{Linearized dislocation mechanics}\label{sec:lin_disloc}
We consider the following `small deformation' model of nonlinear dislocation mechanics:
\begin{equation}\label{eq:sys}
    \begin{aligned}
    \p_j v_i - \p_t u_{ij} - J_{ij} & = 0\\
    - \rho \, \p_t v_i + \p_j T_{ij} & = 0\\
    \mathbb{C}_{ijkl} u_{kl} & =  T_{ij} \\
    e_{jrs}\alpha_{ir} V_{s} & =  J_{ij} \\
    e_{rjk}\, \p_j u_{ik} & = \alpha_{ir} \\
    e_{smn} T_{pm} \alpha_{pn} & = V_s
    \end{aligned}
\end{equation}
where $v$ is the material particle velocity, $u = \veps + \omega$ is the elastic distortion, $\veps_{ij} = \frac{1}{2}(u_{ij} + u_{ji})$ is the elastic strain, $\omega_{ij} = \frac{1}{2}(u_{ij} - u_{ji})$ is the elastic rotation, $J$ is the dislocation flux or plastic strain rate, $\rho$ is the mass density, $T$ is the stress, $\mathbb{C}$ is the tensor of elastic moduli with major and minor symmetries, $\alpha$ is the dislocation density tensor, and $V$ is the dislocation velocity vector.

 For any $\alpha$ field that is the $\curl$ of a skew symmetric tensor field, and with $0$-traction boundary conditions (i.e., $T_{ij} n_j = 0$, where $n_j$ is the outward unit normal field on the body), the unique solution (in a simply connected body) for stress $T_{ij} = \fo{C}_{ijkl} u_{kl} = \fo{C}_{ijkl} \veps_{kl}$ corresponding to the system
\begin{equation*}
    \begin{aligned}
       \p_j (\fo{C}_{ijkl} u_{kl}) &= 0\\
       e_{rjk}\p_j u_{ik} &= \alpha_{ir}
    \end{aligned}
\end{equation*}
 is $T = 0$ with $\veps = 0$. Thus, for the entire class of such dislocation density distributions, $J = 0, V = 0$, and the class $\left\{(v_i, u_{ij}) | v_i = 0, u_{ij} = \frac{1}{2}(u_{ij} - u_{ij}), \,  \mbox{i.e., skew}\right\}$ constitutes steady state solutions of \eqref{eq:sys}. Motivated by the above, we linearize \eqref{eq:sys} about any state $(v_i, u_{ij})$ with $(v_i = 0, \veps_{ij} = 0)$, with the field  $\alpha$ defined from \eqref{eq:sys}$_5$, regardless of the topology of the body. Such 0-stress steady state solutions represent (within the confines of the `small deformation' theory \eqref{eq:sys} being considered) non-trivial elastic distortion and dislocation density fields (e.g.~arbitrarily fine distributions of dislocation walls) - unlike in classical linear elasticity.  

\textit{Henceforth, we will assume $T = 0, J = 0$, and $V = 0$}.

We denote the dependent fields of the linearized system with overhead dots `$\dot{(\  )}$' to obtain the system
\begin{equation}
\label{eq:lin_sys}
    \begin{aligned}
    \p_j \lf{v}_i - \p_j \lf{u}_{ij} - \lf{J}_{ij} & = 0 \\
    - \rho \p_t \lf{v}_i + \p_j ( \mathbb{C}_{ijkl} \lf{\veps}_{kl} ) & = 0\\
    \mathbb{C}_{ijkl} \lf{\veps}_{kl} & = \lf{T}_{ij}\\
    e_{jrs} \alpha_{ir} \lf{V}_s  & = \lf{J}_{ij}\\
    e_{rjk} \p_j \lf{u}_{ik} & = \lf{\alpha}_{ir}\\
     e_{smn} \lf{T}_{pm} \alpha_{pn} & = \lf{V}_s.
    \end{aligned}
\end{equation}
We then have the following energy equality for the linearized system:
\begin{equation}
\label{eq:en_eq}
    \begin{aligned}
    & \mathbb{C}_{ijkl} \lf{\veps}_{kl} \left( \p_j \lf{v}_ i - \p_j \lf{u}_{ij} - \lf{J}_{ij} \right) = 0,\\
    & - \rho \lf{v}_i \p_t \lf{v}_i + \lf{v}_i \p_j ( \mathbb{C}_{ijkl} \lf{\veps}_{kl} ) = 0\\
    & \Longrightarrow \p_j \left( \lf{v}_i \mathbb{C}_{ijkl} \lf{\veps}_{kl} \right) - \p_t \left( \frac{1}{2} \rho \lf{v}_i \lf{v}_i + \frac{1}{2} \mathbb{C}_{ijkl} \lf{\veps}_{ij} \lf{\veps}_{kl} \right)  =  \mathbb{C}_{ijkl} \lf{\veps}_{kl} \lf{J}_{ij}  =  \lf{T}_{ij} \lf{J}_{ij}.
    \end{aligned}
\end{equation}
From \eqref{eq:lin_sys}$_{4,6}$ we have
\begin{equation}\label{eq:diss_sign}
    \begin{aligned}
    \lf{V}_s & = e_{sjr} \lf{T}_{ij} \alpha_{ir} \\
    \lf{T}_{ij} \lf{J}_{ij} & =  \lf{T}_{ij} e_{sjr} \alpha_{ir} \lf{V}_s = \lf{V}_s \lf{V}_s \geq 0,
    \end{aligned}
\end{equation}
which yields energy dissipation and will be crucial to our analysis.
Introducing the notation
\begin{equation}\label{eq:lin_J}
    \begin{aligned}
    \lf{J}_{ij} & = e_{sjr} \alpha_{ir} e_{smn} \lf{T}_{pm} \alpha_{pn} = (e_{sjr} \alpha_{ir} e_{smn}\alpha_{pn}) \mathbb{C}_{pmkl} \lf{u}_{kl} = \fo{B}_{ijkl} \lf{u}_{kl} = \fo{B}_{ijkl} \lf{\veps}_{kl}
    \end{aligned}
\end{equation}
\eqref{eq:lin_sys} becomes
\begin{equation}\label{eq:gov_eqn}
    \begin{aligned}
    \p_t \lf{u}_{ij} - \p_j \lf{v}_i + \fo{B}_{ijkl} \lf{u}_{kl} & = 0\\
    \rho \p_t \lf{v}_i - \p_j (\fo{C}_{ijkl} \lf{u}_{kl} ) & = 0,
    \end{aligned}
\end{equation}
with `nonnegative dissipation' \eqref{eq:diss_sign} in force.
\section{Analysis of linearized dislocation mechanics}\label{sec:anal}
For the analysis of system \eqref{eq:lin_sys} it is useful to decompose it into symmetric and skew parts as follows
\begin{equation}\label{eq:lin_sym_sys}
    \begin{aligned}
         \p_t \lf{\veps}_{ij} & = \frac{1}{2} (\p_j \lf{v}_i + \p_i \lf{v}_j) - \lf{J}_{(ij)}\\
         \rho \p_t \lf{v}_i & = \p_j \fo{C}_{ijkl} \lf{\veps}_{kl} \\
         \p_t \lf{\omega}_{ij} & = \frac{1}{2} (\p_j \lf{v}_i - \p_i \lf{v}_j) - \lf{J}_{[ij]}.
    \end{aligned}
\end{equation}
Here, and in the following, we use the index notation $A_{\ldots(ij)\ldots} = \frac{1}{2} ( A_{\ldots ij \dots} + A_{\ldots ji \ldots})$ and $A_{\ldots [ij] \dots} = \frac{1}{2} ( A_{\ldots ij \ldots} - A_{\ldots ji \ldots})$.

We note from \eqref{eq:lin_J} that $\fo{B}_{ijkl} = \fo{B}_{ij(kl)}$, a property inherited from the minor symmetries of the elastic moduli $\fo{C}$, and
\[
\lf{J}_{ij} = \fo{B}_{ij(kl)}\lf{\veps}_{kl}.
\]
\subsection{Existence and Uniqueness}
Let $\Omega$ be an open, bounded, connected set in $\mathbb{R}^3$. By $\p \Omega$ we denote the boundary of $\Omega$ which is $C^1$-smooth and by $n_j, j = 1,2,3$ the unit normal on $\p \Omega$ pointing towards the exterior. We denote by $\p \Omega_1, \p \Omega_2$ fixed subsets of $\p \Omega$, and $\p \Omega_1^c := \p \Omega - \overline{\p \Omega_1}$, $\p \Omega_2^c := \p \Omega - \overline{\p \Omega_2}$.

Introduce the Sobolev spaces
\begin{equation*}
    \begin{aligned}
    L^2(\Omega) & = \left\{ v:\Omega \to \fo{R}^3   \,\vert \, \int_\Omega dx \, v_i v_i < \infty \right\}\\
    H^0(\Omega) & = \left\{\veps:\Omega \to \fo{R}^{3 \times 3}_{sym}   \, | \, \int_\Omega dx \, \veps_{ij} \veps_{ij}  < \infty \right\} \ \fo{R}^{3 \times 3}_{sym}\ \mbox{is the set of $3 \times 3$ symmetric matrices}\\
    H^1(\Omega) & = \left\{v:\Omega \to \fo{R}^3  \, | \, \p_j v_i \, \mbox{in the sense of distributions}, \, \int_\Omega dx \, \p_j v_i \p_j v_i + v_i v_i < \infty \right\},
    \end{aligned}
\end{equation*}
and we insist on
\begin{equation*}
    \begin{aligned}
    \fo{C}_{ijkl} & = \fo{C}_{jikl} = \fo{C}_{ijlk} = \fo{C}_{klij} \\
    \fo{C}_{ijkl} \veps_{ij}\veps_{kl} & \geq a_0 \,\veps_{ij} \veps_{kl}, \forall \veps, a_0 > 0 \  \mbox{ a constant}. 
    \end{aligned}
\end{equation*}
Let 
\[
U = (\lf{\veps}_{ij},\lf{v}_i)
\]
so that \eqref{eq:lin_sys} is written as
\begin{equation}\label{eq:sys_not}
    \begin{aligned}
    \p_t U & = A\, U\\
    A\, U & = \begin{bmatrix}
    \p_{(j} \lf{v}_{i)} - \fo{B}_{(ij)(kl)} \lf{\veps}_{kl}\\
    \\
    \frac{1}{\rho} \p_j (\fo{C}_{ijkl} \lf{\veps}_{kl})\\
    \end{bmatrix}.
    \end{aligned}
\end{equation}
We wish to establish that $A$ is the infinitesimal generator of a $C^0$ semi-group of contractions $S(t)$ on the Hilbert space
\begin{equation*}
    H = H^0(\Omega) \times L^2(\Omega)
\end{equation*}
with inner product $\langle \ , \ \rangle_H$ defined as
\begin{equation}\label{eq:diss_func}
    \begin{aligned}
    \langle (\veps_{ij}, v_i), (\bar{\veps}_{ij}, \bar{v}_i) \rangle_H & = \langle \fo{C}_{ijkl} \veps_{kl}, \bar{\veps}_{ij} \rangle_{H^0(\Omega)} + \langle \rho v_i, \bar{v}_i \rangle_{L^2(\Omega)}\\
    D(A) & = \{ (\veps_{ij}, v_i) \in H | \, AU \in H, v_i(x) = 0, \mbox{a.e.}\, x \in \p \Omega_1, \fo{C}_{ijkl} \veps_{kl} n_j = 0 \  a.e.~x \in \p \Omega_2 \}\\
    & = \left\{(\veps_{ij}, v_i) \in H | \, v_i \in \widehat{H}^1(\Omega), \p_j (\fo{C}_{ijkl} \veps_{kl}) \in L^2(\Omega), \fo{C}_{ijkl} \veps_{kl} n_j = 0 \  a.e.~ x\in \p \Omega_2 \right\}\\
    \widehat{H}^1(\Omega) &:= \left\{ v_i \in H^1(\Omega) | v_i = 0 \ a.e~ x\in \p \Omega_1 \right\}. 
    \end{aligned}
\end{equation}

We note that $\overline{D(A)} = H$. Furthermore, we have $\langle U, AU \rangle_H = \langle AU, U\rangle_H \leq 0$ for all $U \in D(A)$, i.e., $A$ is \emph{dissipative}. To see this, we compute
\begin{equation*}
    \begin{aligned}
    \langle U, AU \rangle_H & = \int_\Omega dx \,  \fo{C}_{klij} \lf{\veps}_{ij} (\p_{(l} \lf{v}_{k)} - \fo{B}_{(kl)(pm)} \lf{\veps}_{pm}) + \lf{v}_i \p_j (\fo{C}_{ijkl} \lf{\veps}_{kl}) \\
    & = \int_{\p \Omega} da \, \lf{v}_i \fo{C}_{ijkl} \lf{\veps}_{kl} n_j - \int_\Omega dx \, \fo{C}_{klij} \lf{\veps}_{ij} \fo{B}_{(kl)(pm)} \lf{\veps}_{pm} \\
    & = - \int_\Omega dx \, \lf{T}_{kl} \lf{J}_{kl} = - \int_\Omega dx \, \lf{V}_s \lf{V}_s \leq 0.
    \end{aligned}
\end{equation*}
By the Lumer-Phillips Theorem \cite[p.~250]{yosida} $A$ will then be the infinitesimal generator of a $C^0$ semigroup of contractions on $H$ if the range condition
\begin{equation*}
    \mbox{range}(\lambda I - A) = H,
\end{equation*}
for some $\lambda \in \mathbb{R}$, i.e., we need solvability of the system
\begin{equation}\label{eq:solvability}
    \begin{aligned}
    \lambda \lf{\veps}_{ij} - \p_{(j} \lf{v}_{i)} + \fo{B}_{(ij)(kl)} \lf{\veps}_{kl} & = f_{ij} \\
     \lambda \rho \lf{v}_i - \p_j ( \fo{C}_{ijkl} \lf{\veps}_{kl} ) & = g_i,
    \end{aligned}
\end{equation}
for given $(f,g) \in H$.

We ensure this by writing for $\lambda$ sufficiently large
\[
\lf{\veps}_{ij} = \fo{G}_{ijkl} (\p_l \lf{v}_k + f_{kl} )
\]
where $\fo{G}_{ijkl}$ is the inverse of $(\lambda \delta_{ik} \delta_{jl} + \fo{B}_{(ij)(kl)})$, or more succintly, with $\fo{D}_{ijkl} := \fo{B}_{(ij)(kl)}$, $\fo{G} = \mathcal{R}(- \lambda, \fo{D})$ where $\mathcal{R}(\lambda, \fo{D})$ denotes the resolvent of $\fo{D}$.

The usual expansion of the resolvent (see, e.g., \cite[p.~211]{yosida}) gives
\[
\lambda \mathcal{R}(-\lambda, \fo{D}) = \fo{I} + \fo{D} \lambda^{-1} \sum_{j = 0}^\infty (-\lambda)^{-j} \fo{D}^j.
\]
Hence for $r = \frac{\| \fo{D} \|}{\lambda} < 1$, where $\| \fo{D} \|$ denotes the operator norm of $\fo{B}$, the series converges and we define
\begin{equation}
  \begin{aligned}\label{eq:def_K}
  \fo{K} & = \fo{D} \lambda^{-1} \sum_{j = 0}^\infty (-\lambda)^{-j} \fo{D}^j\\
  \| \fo{K} \| & = r(1-r)
  \end{aligned}
\end{equation}
to get
\[
\lambda \fo{G} = \fo{I} + \fo{K} \Longrightarrow \fo{G} = \lambda^{-1} \fo{I} + \lambda^{-1} \fo{K},
\]
where $\| \lambda^{-1} \fo{K} \| \leq \mbox{const}\, \lambda^{-2}$. Hence $\fo{G}_{ijkl} = \lambda^{-1} \delta_{ik} \delta_{jl} + \lambda^{-1} \fo{K}_{ijkl}$, and
\begin{equation}\label{eq:u_invert}
    \lf{\veps}_{ij} = \lambda^{-1} (\p_{(j} \lf{v}_{i)} + f_{ij}) + \lambda^{-1} \fo{K}_{ijkl} (\p_{(l} \lf{v}_{k)} + f_{kl})
\end{equation}
Now substitute \eqref{eq:u_invert} into \eqref{eq:solvability}  to obtain
\begin{equation}\label{eq:strong}
\begin{aligned}
   & \lambda^2 v_i - \p_j (\fo{C}_{ijkl} (\p_l \lf{v}_k + f_{kl})) - \p_j (\fo{C}_{ijkl} \fo{K}_{klmn} (\p_n \lf{v}_m + f_{mn}))  = \lambda g_i \\
    & \Longrightarrow \lambda^2 v_i - \p_j (\fo{C}_{ijkl} \p_l \lf{v}_k) - \p_j (\fo{C}_{ijkl}\fo{K}_{klmn} \p_n \lf{v}_m) = \p_j (\fo{C}_{ijkl} f_{kl}) + \p_j (\fo{C}_{ijkl}\fo{K}_{klmn} f_{mn})+ \lambda g_i
\end{aligned}
\end{equation}
Write \eqref{eq:strong} in weak form
\begin{equation}\label{eq:weak_form}
\begin{aligned}
     \int_\Omega dx \, \lambda^2 \overline{v}_i \lf{v}_i + \p_j\overline{v}_i  \fo{C}_{ijkl} \p_l \lf{v}_k & + \p_j \overline{v}_i (\fo{C}_{ijkl} \fo{K}_{klmn}) \p_n \lf{v}_m  \\
    & =  \int_\Omega dx \, - \p_j  \overline{v}_i \fo{C}_{ijkl} f_{kl} - \p_j \overline{v}_i (\fo{C}_{ijkl}\fo{K}_{klmn} ) f_{mn} + \lambda \overline{v}_i g_i \\
     & \qquad \forall \ \overline{v}_i \in \widehat{H}^1(\Omega).
\end{aligned}
\end{equation}
The left hand side of \eqref{eq:weak_form} defines a bilinear form $\mathcal{B}(v, \overline{v})$ which satisfies
\begin{enumerate}
    \item $\mathcal{B}(v, \overline{v}) \leq \mbox{const} \| v\|_{\widehat{H}^1(\Omega)}\|\overline{v}\|_{\widehat{H}^1(\Omega)}$  (boundedness) 
    \item $\mathcal{B}(v, v) \geq \mbox{const} \| v\|^2_{\widehat{H}^1(\Omega)}$ for $\lambda$ sufficiently large by \eqref{eq:def_K} (coercivity).
\end{enumerate}
Here we have assumed that $\fo{C}_{ijkl}, \alpha_{ij}$ are continuous on $\Omega$ and used Korn's inequality (see, e.g, Ciarlet \cite{ciarlet2010korn}).

As the right hand side of \eqref{eq:weak_form} defines a bounded linear functional on $\widehat{H}^1(\Omega)$, the Lax-Milgram Theorem \cite[p.~92]{yosida} yields a unique solution $\lf{v}_i \in \widehat{H}^1(\Omega)$ to the weak form \eqref{eq:weak_form}. Substituting this $\lf{v}_i$ into \eqref{eq:u_invert} we have defined $\lf{\veps}_{ij}$ where the pair
\[
(\lf{\veps}_{ij}, \lf{v}_i) \in H^0(\Omega) \times \widehat{H}^1(\Omega)
\]
solves \eqref{eq:solvability}. From \eqref{eq:solvability} we see that $\p_j (\fo{C}_{ijkl} \lf{\veps}_{kl}) \in L^2(\Omega)$. Furthermore, retracing our steps from \eqref{eq:weak_form} back to \eqref{eq:solvability} tells us that 
\begin{equation*}
    \int_{\p \Omega_2} \overline{v}_i  \fo{C}_{ijkl} \lf{\veps}_{kl}\, n_j \, da = 0
\end{equation*}
and hence $(\lf{\veps}_{ij}, \lf{v}_i) \in D(A)$.
Thus, the range condition of the Lumer-Phillips theorem is satisfied and we we have
\begin{theorem}
$A$ is the infinitesimal generator of a $C^0$ semi-group of contractions $S(t)$ on $H$and for $U_0 \in H$, $S(t)U_0$ provides the unique weak solution to \eqref{eq:lin_sym_sys}$_{1,2}$.
\end{theorem}
\begin{corollary}\label{corr}
$(S(t)U_0, \lf{\omega}_{ij})$ yields the unique weak solution of \eqref{eq:lin_sym_sys}, where $\lf{\omega}_{ij}$ is defined via the definite integral in $t$ of the right hand side of \eqref{eq:lin_sym_sys}$_3$ with initial condition $\lf{\omega}_{ij}(0) \in H^0(\Omega)$ at $t = 0$.
\end{corollary}
\begin{proof}
Substitute $S(t)U_0 = (\lf{\omega}_{ij}(t), \lf{v}_i(t))$ into the well-defined right hand side of \eqref{eq:lin_sym_sys}$_3$ yielding a continuous in time right hand side.
\end{proof}
\subsection{Time-asymptotic behavior}
Since our system \eqref{eq:lin_sym_sys}$_{1,2}$ is linear, we are able to exploit the special properties of linear contraction semigroup on Hilbert space. As it is not obvious that $(\lambda I - A)^{-1}: H \to H$ is compact, we cannot immediately apply the results of Dafermos and Slemrod \cite{dafermos1973asymptotic} on strong decay in $H$. Instead, we follow an argument of O'Brien \cite{o1978contraction} which in turn was based on the following theorem of S. Foguel \cite{foguel1966ergodic}:
\begin{theorem}
For a $C^0$ semigroup of contractions $S(t)$ on a Hilbert space $H$ define the isometric subspace $H_u$ of $H$ as
\begin{equation*}
    H_u = \left\{ U_1 \in H \ | \ \| S(t) U_1 \|_H = \|U_1 \|_H = \left\| S^*(t) U_1 \right\|_H, t \geq 0\right\}.
\end{equation*}
Then $H_u$ is a closed invariant subspace and $S(t)$ forms a $C^0$ semi-group of unitary operators on $H_u$, and for $W_0$ orthogonal to $H_u$
\begin{equation}
    \label{eq:orth_weak}
    S(t)W_0 = S^*(t)W_0 \rightharpoonup 0
\end{equation}
as $t \to \infty$, where $\rightharpoonup$ denotes weak convergence in $H$.
\end{theorem}

We apply Foguel's theorem as follows:

Decompose
\begin{equation*}
\begin{aligned}
    & A = A_1 + A_2 \\
    & (A_1 U)_{ij} = \begin{bmatrix}
    \p_{(j} \lf{v}_{i)} \\
    \\
    \frac{1}{\rho} \p_j (\fo{C}_{ijkl} \lf{\veps}_{kl})
    \end{bmatrix}
    \qquad
    (A_2 U)_{ij} = \begin{bmatrix}
     - \fo{B}_{(ij)(kl)} \lf{\veps}_{kl} \\
     \\
     0
     \end{bmatrix}
\end{aligned}
\end{equation*}
with $D(A_1) = D(A)$. Recall from the dissipation inequality \eqref{eq:diss_sign} and \eqref{eq:diss_func} that 
\begin{equation}\label{eq:diss_ta}
    \langle U, A U \rangle = \langle U , A_2 U \rangle = - \| \lf{V} \|^2_{L^2(\Omega)}.
\end{equation}
Now decompose the initial data $U_0 \in H$ as $U_0 = U_1 + W_0$ with $U_1 \in H_u$ and $W_0 \perp H_u$. By Foguel's theorem
\begin{equation}
    \label{eq:lim_U0}
    \left(S(t)U_0 - S(t)U_1 \right) \rightharpoonup 0 \qquad \mbox{as} \ t \to \infty.
\end{equation}
Hence time-asymptotic behavior of $S(t)U_0$ is determined by time-asymptotic behavior of $S(t)U_1$. Furthermore, since $H_u$ is invariant under $S(t)$ and $D(A)$ is dense in $H$, then $H_u \cap D(A)$ is dense in $H_u$. Take $U_1 \in H_u \cap D(A)$ so that $S(t)U_1 \in D(A)$. We know from Foguel's theorem that
\begin{equation}
    \label{eq:rate_StU1}
    \| S(t) U_1 \|^2_H = \| U_1 \|^2_H \qquad \Longrightarrow \qquad \frac{d}{dt} \| S(t) U_1 \|^2_H = 0.
\end{equation}
On the other hand, the dissipation inequality \eqref{eq:diss_ta} gives us
\begin{equation}
    \label{eq:diss_ta_2}
    \frac{d}{dt} \| S(t) U_1 \|^2_H = \langle S(t)U_1, A_2 S(t)U_1 \rangle_H = - \|\lf{V} \|^2_{L^2(\Omega)},
\end{equation}
where
\begin{equation}\label{eq:def_D}
    \lf{V}_s = e_{smn} \alpha_{pn} \fo{C}_{pmkl} \lf{\veps}_{kl} =: \tho{D}_{skl} \lf{\veps}_{kl} = : (\tho{D} U)_s,
\end{equation}
i.e.,
\begin{equation}
    \label{eq:rate_StU1_2}
    \frac{d}{dt} \| S(t) U_1 \|^2_H = - \| \tho{D} S(t) U_1 \|^2_{L^2(\Omega)} \qquad \forall \ t \in \mathbb{R}^+.
\end{equation}
Comparison of \eqref{eq:rate_StU1} and \eqref{eq:rate_StU1_2} gives
\begin{equation}
    \label{eq:diss_StU1}
    \tho{D}S(t)U_1 = 0 \qquad \forall \ t \geq 0.
\end{equation}
But if $\tho{D}S(t)U_1 = 0$ for $ t \geq 0$ then $A_2 S(t) U_1 = 0$ for $t \geq 0$ and our system reduces to
\[
\frac{d}{dt} U = A_1 U,
\]
$U = S(t) U_1$. But $A_1$ generates the unitary semigroup $S_1(t)$, i.e., $S_1(t)U_1$ is the unique weak solution to the system of linear elasticity where $\lf{v}_i$ now plays the role of the displacement. Hence
\begin{equation}
    \label{eq:StU1=S1tU1}
    S(t)U_1 = S_1(t)U_1, \qquad t \geq 0,
\end{equation}
and from \eqref{eq:diss_StU1}
\begin{equation}
    \label{eq:diss_S1tU1}
    DS_1(t)U_1 = 0, \qquad t \geq 0.
\end{equation}
By density of $H_u \cap D(A)$ in $H_u$ \eqref{eq:StU1=S1tU1} and \eqref{eq:diss_S1tU1} hold for all $U_1 \in H_u$. Hence we have proven
\begin{theorem}\label{thm:limit}
$(S(t)U_0 - S_1(t) U_1 ) \rightharpoonup 0$ as $t \to \infty$ when $\tho{D}S_1(t)U_1 = 0$ for $t \geq 0$, i.e. weak solutions of our system \eqref{eq:lin_sym_sys}$_{1,2}$ weakly approach weak solutions of the equations of \textit{linear elasticity} which are \textit{constrained} by 
$e_{sjr} \alpha_{ir} \lf{T}_{ij} = e_{sjr} \alpha_{ir} \fo{C}_{ijkl} \lf{\veps}_{kl} = 0$. The limit system is given by
\begin{equation}\label{eq:limit}
\begin{aligned}
    \p_t \overline{\veps}_{ij} & = \frac{1}{2} (\p_j \overline{v}_i + \p_i \overline{v}_j)\\
    \rho \, \p_t \overline{v}_ i & = \p_j (\fo{C}_{ijkl} \overline{\veps}_{kl})\\
    0  & = e_{smn} \alpha_{pn} \fo{C}_{pmkl} \overline{\veps}_{kl} \ \mbox{in} \ \Omega,\\
    \mbox{and} \quad \overline{v}_i &= 0 \quad \mbox{on} \ \p \Omega_1\\
    \fo{C}_{ijkl} \overline{\veps}_{kl} n_j &= 0 \quad \mbox{on} \ \p \Omega_2.
\end{aligned}
\end{equation}
\end{theorem}
\begin{corollary} $\p_t \lf{\omega}_{ij}(t) - \frac{1}{2} (\p_j \overline{v}_i - \p_i \overline{v}_j ) (t) \to 0$ as $t \to \infty$ in the sense of distributions, where $\overline{\epsilon}_{ij}, \overline{v}_i$ satisfies \eqref{eq:limit}.
\end{corollary}

\begin{proof}
We know from Theorem \ref{thm:limit} that
\begin{equation*}
    (\lf{\veps}(t), \lf{v}(t)) - (\overline{\veps}(t), \overline{v}(t)) \rightharpoonup 0
\end{equation*}
in $H$ as $t \to \infty$, where $(\overline{\veps}, \overline{v})$ satisfy \eqref{eq:limit}. Hence for all $w \in C^\infty_0(\Omega)$ we have
\begin{equation*}
\begin{aligned}
     \langle \p_{[j} \lf{v}_{i]}(t), w \rangle_{L^2(\Omega)} -  \langle \p_{[j}\overline{v}_{i]}(t), w \rangle_{L^2(\Omega)} & \to 0 \qquad \mbox{as} \ t \to \infty \\
     \langle \lf{J}_{[ij]} (t) , w \rangle_{L^2(\Omega)} & \to 0 \qquad \mbox{as} \ t \to \infty.
\end{aligned}
\end{equation*}
That is 
\begin{equation*}
\begin{aligned}
\p_{[j} \lf{v}_{i]}(t) - \p_{[j} \overline{v}_{i]}(t) & \to 0 \\
\lf{J}_{[ij]}(t) & \to 0  \qquad \mbox{as} \ t \to \infty
\end{aligned}
\end{equation*}
in the sense of distributions.

By Corollary \ref{corr}, \eqref{eq:lin_sym_sys}$_3$ is satisfied in the sense of distributions for $\lf{\omega}(0) \in H^0(\Omega)$. Hence for \eqref{eq:lin_sym_sys}$_3$ we have
\begin{equation*}
    \p_t \lf{\omega}_{ij}(t) - \frac{1}{2} \p_{[j} \overline{v}_{i]}(t) \to 0  \qquad \mbox{as} \ t \to \infty
\end{equation*}
in the sense of distributions, where $(\overline{\veps}_{ij}, \overline{v}_i)$ satisfy \eqref{eq:limit}.
\end{proof}

\subsection{Limit system analysis}
We now provide an analysis of the limit system \eqref{eq:limit} following the argument of Dafermos \cite{dafermos1968existence}.

First assume that the data for the limit system is in $D(A^2_1) = \{(\overline{\veps}_{ij}, \overline{v}_i) \in H | A_1(\overline{\veps}_{ij}, \overline{v}_i) \in H \}$ so that we may differentiate w.r.t $t$ to obtain $\rho \p_{tt} \overline{v}_i = \p_j (\fo{C}_{ijkl} \p_t \overline{\veps}_{kl})$, i.e.,
\begin{equation} \label{eq:lim_sys_anal}
    \rho \p_{tt} \overline{v}_i = \p_j (\fo{C}_{ijkl} \p_l \overline{v}_k)  \qquad \mbox{in} \  \Omega, t \in \mathbb{R}^+, 
\end{equation}
with $\overline{v}_i = 0$ on $\p \Omega_1$, $(\fo{C}_{ijkl} \p_l \overline{v}_k) n_j = 0$ on $\p \Omega_2$ (the restriction on the data is only for temporary convenience since our results for the general data in $H$ follow by density of $D(A^2_1)$ in $H$).

Express solutions of \eqref{eq:lim_sys_anal} as
\begin{equation}
    \label{eq:eig_exp}
    \overline{v}_i = \sum_p c^{(p)} e^{i \lambda^{(p)} t} \phi^{(p)}_i, \qquad c^{(p)} \ \mbox{constants}.
\end{equation}
Then,  \eqref{eq:lim_sys_anal} implies
\begin{equation*}
   \sum_p - \rho {\lambda^{(p)}}^2 c^{(p)} e^{i \lambda^{(p)} t} \phi^{(p)}_i = \sum_p c^{(p)} e^{i \lambda^{(p)} t} \p_j \left(\fo{C}_{ijkl} \p_l \phi_k^{(p)}\right)
\end{equation*}
with $\phi^{(p)}_l = 0$ on $\p \Omega_1$, $(\fo{C}_{ijkl} \p_l \phi_k) n_j = 0$ on $\p \Omega_2$.

Hence, for eigenfunctions $\phi^{(p)}_i$ satisfying
\begin{equation}\label{eq:eig_blm}
    \p_j \left( \fo{C}_{ijkl} \p_l \phi^{(p)}_k \right) = - \rho {\lambda^{(p)}}^2 \phi_i^{(p)} \  \mbox{in} \ \Omega \ \mbox{with} \ \phi^{(p)}_i = 0 \ \mbox{on} \ \p \Omega_1, \  (\fo{C}_{ijkl} \p_l \phi_k) n_j = 0 \ \mbox{on} \ \p \Omega_2,
\end{equation}
\eqref{eq:eig_exp} yields a solution of \eqref{eq:limit}$_{1,2}$. 

The eigenfunctions exist and are orthogonal in $H^0(\Omega)$. Furthermore, we may assume they are normalized so that $\lVert \phi \rVert_{H^0(\Omega)} = 1$. Thus, $c^{(p)}$ may be determined by initial data $\overline{v}_i (0) = \sum_p c^{(p)} \phi^{(p)}_i$, i.e.~ $c^{(p)} = \left( \overline{v}(0), \phi^{(p)} \right)_{H^0(\Omega)}$.

Note that since \eqref{eq:limit}$_3$ holds for all $t \in \mathbb{R}$, we may differentiate w.r.t $t$ to obtain $e_{smn} \alpha_{pn} \fo{C}_{pmkl} \p_t \overline{\veps}_{kl} = 0$ for all $t \in \mathbb{R}$, and hence by \eqref{eq:limit}
\begin{equation}
    \label{eq:vel_constr}
    e_{smn} \alpha_{pn} \fo{C}_{pmkl} \p_l \overline{v}_k = 0.
\end{equation}
Substitute \eqref{eq:eig_exp} into \eqref{eq:vel_constr} to obtain
\begin{equation}\label{eq:vel_constr_eig}
    e_{smn} \alpha_{pn} \fo{C}_{pmkl} \sum_q c^{(q)} e^{i \lambda^{(q)} t} \p_l \phi^{(q)}_k = 0 \  \mbox{in} \ \Omega \ \mbox{for all} \  t \in \mathbb{R}^+.
\end{equation}

We assume there are no repeated eigenvalues for our domain $\Omega$. Then the lhs of \eqref{eq:vel_constr_eig} defines an almost periodic function and we must have
\begin{equation}
    \label{eq:vel_constr_eig_stat}
     e_{smn} \alpha_{pn} \fo{C}_{pmkl} \, c^{(q)} \, \p_l \phi^{(q)}_k = 0 \qquad \mbox{no sum on} \ q.
\end{equation}

Thus there are two cases to consider:

\underline{Case 1}:
\begin{equation*}
    e_{smn} \alpha_{pn} \fo{C}_{pmkl}\, \p_l \phi^{(q)}_k = 0 \qquad \mbox{for some} \ q.
\end{equation*}
In this case $c_q \neq 0$ for this choice of $q$ and the limit solution \eqref{eq:limit} may contain non-trivial oscillations.

\underline{Case 2}:
\begin{equation*}
    e_{smn} \alpha_{pn} \fo{C}_{pmkl}\, \p_l \phi^{(q)}_k \neq 0 \qquad \mbox{for all} \ q.
\end{equation*}
In this case $c_q = 0$ for all $q$ and we have
\begin{equation*}
    \overline{v_i} = 0 \ \mbox{in} \ \Omega \ \mbox{for all} \ t 
\end{equation*}
and hence by \eqref{eq:limit}
\begin{equation*}
    \p_t \overline{\veps}_{ij} = 0 \ \mbox{in} \ \Omega \ \mbox{for all} \ t 
\end{equation*}
i.e.~$\overline{\veps}_{ij}$ is a time-independent equilibrium solution of \eqref{eq:limit} with
\begin{equation}
    \label{eq:stat_case}
    \begin{aligned}
    \p_j (\fo{C}_{ijkl} \overline{\veps}_{kl} ) & = 0 \\
    e_{smn} \alpha_{pn} \fo{C}_{pmkl} \overline{\veps}_{kl} & = 0
    \end{aligned}
\end{equation}
in $\Omega$.
We provide two explicit examples corresponding to Cases 1 and 2.

\noindent \underline{Examples}

Recall that we need the base $\alpha_{ij}$ state to be the curl of a skew-symmetric tensor field to have $T_{ij} = 0$, i.e.~a stress-free base state. Thus, $\alpha_{ij}$ must be of the form
\begin{equation*}
    \alpha_{ij} = e_{jmk} \p_m (e_{ikr} \psi_r) = -(\p_r \psi_r) \delta_{ij} + \p_i \psi_j.
\end{equation*}
We choose the vector field $\psi$ of the form
\begin{equation*}
    \psi_k (x_1, x_2, x_3) = A(x_1, x_2, x_3) \delta_{k1}
\end{equation*}
which yields
\begin{equation*}
    \alpha_{pn} = -(\p_1 A) \delta_{pn} + \p_p A \delta_{n1}.
\end{equation*}
This distribution corresponds to a `crossed-grid' of screw dislocations parallel to the coordinate axes superposed on an edge dislocation distribution with line direction along $x_1$ and Burgers vector in the $x_2-x_3$ plane.

Furthermore, choose the ansatz of uniaxial stress fields
\begin{equation}\label{eq:str_ans}
    \overline{T}_{pm} = \sigma(x_1,t) \delta_{p1} \delta_{m1}
\end{equation}
to note that the r.h.s. of \eqref{eq:limit}$_3$ becomes
\begin{equation*}
    e_{smn}\left( (\p_1 A) \delta_{pn} + (\p_p A) \delta_{n1} \right) \sigma \delta_{p1} \delta_{m1} = 0. 
\end{equation*}
Thus, by our chosen ansatz for $\alpha_{ij}$ and the (time-dependent, for now) limit stress field $\overline{T}_{ij}$, \eqref{eq:limit}$_3$ is satisfied. 

Physically, the screw dislocation distributions along $x_2$ and $x_3$ directions (i.e., $\alpha_{22}, \alpha_{23}$) do not have a Burgers vector favorably aligned to the uniaxial stress field along the $x_1$ direction to produce a non-trivial Peach-Koehler driving force. The edge and screw distributions with line direction in the $x_1$ direction do not lie parallel to a plane on which the uniaxial stress field ansatz $\sigma \delta_{p1} \delta_{m1} = \overline{T}_{pm}$ can produce a traction, and hence these see no driving force for motions as well (the edge distributions with Burgers vector in the $x_2-x_3$ plane share both of these `null' driving force attributes).
\subsubsection{Static solution of limit system}
Consider an isotropic elastic material with $\fo{C}_{ijkl} = \lambda \delta_{ij} \delta_{kl} + \mu (\delta_{ik}\delta_{jl} + \delta_{il}\delta_{jk})$, where $\lambda, \mu$ are the Lam\'e parameters. Consider a time-independent, uniaxial stress field of the form \eqref{eq:str_ans} with $\sigma(x_1)$. Now define $\overline{\veps}_{ij}$ through
\begin{equation*}
    \overline{\veps}_{ij}(x_1,x_2,x_3, t) := \frac{1}{2\mu}\left(\sigma(x_1) \delta_{i1} \delta_{j1} - \frac{\lambda}{3 \lambda + 2 \mu} \sigma(x_1) \delta_{ij} \right)
\end{equation*}
(which can also be expressed as $\fo{S}_{ijkl} \left(\sigma(x_1) \delta_{k1}\delta_{l1}\right)$, where $\fo{S}_{ijkl} = \frac{1}{4 \mu} (\delta_{ik}\delta_{jl} + \delta_{il}\delta_{jk}) -  \frac{\lambda}{2 \mu(3 \lambda + 2 \mu)}\delta_{ij}\delta_{kl}$ is the elastic compliance tensor, the inverse of the elastic stiffness, $\fo{C}_{ijkl}$, on the space of symmetric second-order tensors). Hence, $\fo{C}_{ijkl} \,\overline{\veps}_{kl}(x_1,x_2,x_3, t) = \sigma(x_1) \delta_{i1}\delta_{j1}$ in this case. Making the choice $\sigma(x_1) = \sigma_0$, where $\sigma_0 \in \fo{R}$ is an arbitrary constant, and setting
\begin{equation*}
    \begin{aligned}
    \overline{\veps}_{ij} & = \fo{S}_{ijkl} \sigma_0 \delta_{k1} \delta_{l1} \\
    \overline{v}_i & = 0,
    \end{aligned}
\end{equation*}
we have a non-trivial static solution to the limit system \eqref{eq:limit} with $\p \Omega_2 = \phi$. 
\subsubsection{Oscillating solution of the limit system}
We consider a time-dependent, uniaxial stress field ansatz corresponding to isotropic linear elasticity with shear modulus $\mu$ and the other Lam\'e parameter set to $0$ in terms of the function $U(x_1,t)$ given by
\begin{equation*}
    \overline{T}_{ij} (x_1,t) = \mu ( \p_j U (x_1,t) \delta_{i1} + \p_i U(x_1,t) \delta_{j1} ).
\end{equation*}
Comparison with \eqref{eq:str_ans} yields the definition
\begin{equation*}
    \sigma(x_1,t) := 2 \mu \, \p_1 U(x_1,t).
\end{equation*}
We consider a cylinder with a uniform rectangular cross-section normal to the $x_1$ direction as the body $\Omega$, with free (i.e.traction-free) bounding surfaces perpendicular to $x_2$ and $x_3$. The lateral surfaces of the cylinder correspond to $\p \Omega_2$, and $\p \Omega_1$ comprises the boundaries of the cylinder perpendicular to its axis, $x_1$. Let the $x_1$-coordinate of the planar surfaces comprising $\p \Omega_1$ be $x_l$ and $x_r$. We now set
\begin{equation}\label{eq:osc_ex}
\begin{aligned}
    \overline{v}_i (x_1, x_2, x_3, t) & = \p_t U(x_1,t) \delta_{1i}\\
    \overline{\veps}_{ij} (x_1, x_2, x_3, t) & = \, \p_1 U(x_1, t) \delta_{1i}\delta_{1j},
\end{aligned}
\end{equation}
and require that
\begin{equation*}
\begin{aligned}
    \rho \, \p_t\p_t U &= 2 \mu \, \p_1\p_1 U \quad \mbox{in} \ \Omega \\
    U(x_1,t) & = 0 \Rightarrow \p_t U(x_1,t) = 0 \quad \mbox{for} \  x_1 = x_l, x_r.
\end{aligned}
\end{equation*}
Non-trivial solutions to this wave equation exist and define oscillating solutions to the limit system \eqref{eq:limit} through the definitions \eqref{eq:osc_ex} (with $\fo{C}_{ijkl} \veps_{kl} = \overline{T}_{ij}(x_1,t)$).

In summary, our limit system analysis  yields the following consequence of Theorem \ref{thm:limit}:
\begin{corollary}
Assume \eqref{eq:eig_blm} has no repeated eigenvalues. For initial data $U_0 \in H$ the limit system solution $S(t)U_1$ must lie in either Case 1: non-trivial oscillations or Case 2: a non-trivial solution $\overline{\veps}$ to \eqref{eq:stat_case} representing a static solution.
\end{corollary}

\section*{Acknowledgment}
The work of AA was supported by the grant NSF OIA-DMR \#2021019.
\bibliographystyle{alpha}\bibliography{ref}
\end{document}